\documentclass[letterpaper,11pt]{amsart}

\usepackage{amsfonts}
\usepackage{amsmath}
\usepackage{amssymb}
\usepackage{amsthm}

\newtheorem{thm}{Theorem}[section]

\newtheorem{cor}[thm]{Corollary}
\newtheorem{lem}[thm]{Lemma}
\newtheorem{prop}[thm]{Proposition}

\theoremstyle{definition}

\newtheorem{exam}[thm]{Example}

\numberwithin{equation}{section}

\begin{document}
\author[Najmeh Dehghani]{Najmeh Dehghani}
\address{Department of Mathematics, College of Sciences, Persian Gulf University, Bushehr, Iran, 751-6913817.}
\email{n.dehghani@pgu.ac.ir}
\author{S.  Tariq  Rizvi}
\address{Department of Mathematics, The Ohio State University, Lima, Oh 45804,  USA.}
\email{rizvi.1@osu.edu}

\title{When mutually subisomorphic  Baer modules are isomorphic}

\begin{abstract} The Schr\"{o}der-Bernstein Theorem for sets is well known. The question of whether two subisomorphic algebraic structures are isomorphic to each other, is of interest.
An $R$-module $M$  is said to satisfy  the Schr\"{o}der-Bernstein (or SB) property if any  pair of direct summands of $M$ are isomorphic provided that each one is isomorphic to a direct summand of the other.
A ring $R$ (with an  involution $\star$) is called  a Baer (Baer $\star$-)ring if the right annihilator of every nonempty subset of $R$ is generated by an idempotent (a projection). It is clear that every Baer $\star$-ring is a Baer ring.
Kaplansky showed that  Baer $\star$-rings  satisfy  the SB property. This motivated us to investigate whether any Baer ring  satisfies the SB property. In this paper we carry out a study of this question and investigate when two subisomorphic Baer modules are isomorphic.  Besides, we study extending modules which satisfy the SB property. 
We characterize a commutative domain $R$ over which any  pair of subisomorphic extending modules are isomorphic.

\end{abstract}
\maketitle


\section{Introduction}
The famous Schr\"{o}der-Bernstein Theorem states that any two sets with one to one maps into  each other are isomorphic.
The question of whether two subisomorphic algebraic structures  are  isomorphic to each other has been of interest to a number of researchers. Various analogues of the Schr\"{o}der-Bernstein Theorem have been appeared for categories of associative rings, categories of functors and categories of $R$-modules
 \cite{bumby}, \cite{cantor2}, \cite{S.B}, \cite{Ashish}, \cite{soo.lee}, \cite{Rizvi.Muller}, \cite{rososhek2}, \cite{rososhek},  \cite{cantor1},\cite{Functor},   \cite{vasconcelos} and \cite{correct}.
 Bumby in 1965 \cite{bumby}, showed that any two injective modules which are subisomorphic to each other are isomorphic. M\"{u}ller and Rizvi in 1983 \cite{Rizvi.Muller}, extended Bumby's result for the class of continuous modules which are a generalization of injective modules.

\begin{thm}{\cite[Proposition 10]{Rizvi.Muller}}\label{M. R22}
Let $M$ and $N$ be continuous $R$-modules. If $M$ and $N$ are subisomorphic to each other, then $M\simeq N$.
\end{thm}

They constructed an example which shows that in the above theorem both $N$ and $K$ cannot be quasi-continuous.
For  abelian groups, Kaplansky in 1954 \cite[p.12]{infinite abelian groups},  posed  the following question, also known as  Kaplansky's  First Test Problem:   ``If $G$ and $H$ are abelian groups such that each one is isomorphic to a direct summand of the other,   are $G$ and $H$ necessarily isomorphic?"  Negative answers have been given to this question by several authors \cite{solution of kaplansky}, \cite{the kaplansky test problem} and \cite{I.kaplansky}. Besides Kaplansky in  1968 \cite[Theorem 41]{Kaplansky}, showed that every Baer $\star$-ring satisfies this analogue of the Schr\"{o}der-Bernstein Theorem. Recall that a ring $R$ with an involution $\star$ is called a {\it Baer $\star$-ring} if the
right annihilator of every nonempty subset of $R$ is generated by a projection $e$ (the
idempotent $e$ of the $\star$-ring $R$ is called a {\it projection} if $e^{\star}= e$). In particular he proved the following result:

\begin{thm}\cite[Theorem 41]{Kaplansky}\label{Kap.Baer}
  Let $R$ be a Baer $\star$-ring and   $e$,   $f$  be  projections  in $R$. If  $eR$ is isomorphic to a direct summand of $fR$ and $fR$ is isomorphic to a direct summand of $eR$ then $eR$ is isomorphic to $fR$.
 \end{thm}

 An  $R$-module $M$ is called to {\it satisfy  the Schr\"{o}der-Bernstein property (or SB property)} if any two d-subisomorphic direct summands of $M$ are isomorphic (the $R$-modules $N$ and $K$ are called {\it d-subisomorphic to each other}  whenever $N$ is isomorphic to a direct summand of $K$ and $K$ is isomorphic to a direct summand of $N$).
 Modules which satisfy the SB property was introduced and studied in \cite{S.B}. (For convenience, we have modified the notation in \cite{S.B} from ``S-B" to ``SB").  Moreover a subclass $\mathcal{C}$ of $R$-modules is called to {\it satisfy the SB property} provided that any pair of members are isomorphic whenever they are d-subisomorphic to each other. Modules that satisfy the SB property, have been also studied in   \cite{Ashish}.
 By Kaplansky's  Theorem, every Baer $\star$-ring satisfies the SB property. \\
 Kaplansky in 1968 \cite{Kaplansky}, introduced the notion of Baer ring. Recall that a ring $R$ is called {\it Baer} if the right annihilator of any nonempty subset  of $R$ is generated,  as a right ideal,  by an idempotent.  It is easy to observe that the Baer property is left and right symmetric.
The notion of Baer ring was  extended to a general module theoretic,  introducing a Baer module. An $R$-module $M$ is called {\it  Baer} if for all $N\leq M$, ann$_{S}(N)$ is a direct summand of $S$ where $S={\rm End}_{R}(M)$
 \cite[Chapter 4]{Extension}. Clearly  $R$   is a Baer ring  if $R_R$ is  Baer. Every  nonsingular extending  modules is  Baer \cite[Theorem 4.1.15]{Extension}.

Now what Kaplansky proved for Baer $\star$-rings (Theorem \ref{Kap.Baer}), motivated us to ask ``{\it when  any  pair of subisomorphic or d-subisomorphic Baer   modules are   isomorphic to each other}".\\

In Section 2,
first we give some examples to show that subisomorphic Baer modules are not necessarily  isomorphic to each other (Examples \ref{1.1} and \ref{Baer.subisomorphic}).
We note that  in a Baer $\star$-ring $R$, the set of all projections forms a complete lattice and this is the main point in  the proof of Theorem \ref{Kap.Baer}.  Here, we give an example to show that for Baer rings this is not the case in general (Example \ref{triangular metrix ring}).
Besides, we show that every Baer (or Rickart) module  with only countably many direct summands,  satisfies the SB property (Theorem \ref{Baer.Countable}).
 Moreover, we show that duo Baer rings and reduced Baer rings satisfy the SB property (Theorem \ref{reduced} and Corollary  \ref{duo.Baer.S.B}).\\
 We also investigate rings over which any pair of subisomorphic Baer modules are isomorphic.
  For instance, if $R$ is a  right nonsingular ring  with  finite uniform dimension then   any two subisomorphic Baer $R$-module are isomorphic if and only if $R$ is a semisimple Artinian ring (Corollary  \ref{nonsingular}). Rings  over which every Baer module is injective are precisely semisimple Artinian rings (Theorem \ref{all Baer is inj}). Moreover, we investigate when two  extending modules which are subisomorphic to each other are isomorphic. It is  proved  that the study of the SB property for the class of extending modules  can be reduced to the study of such modules when they are singular and nonsingular (Theorem \ref{separate extendings}). We characterize commutative domains over which any pair of subisomorphic (torsion free) extending modules are isomorphic (Proposition \ref{12} and Corollary \ref{16}).

Throughout this paper, rings have nonzero identity unless  otherwise stated.  All
modules are right and unital.
Let $M$ be an $R$-module. The notations $N\subseteq M$, $N\leq M$, $N\leq_{ess}M$, or $N\leq_{\oplus} M$ mean that $N$ is a subset, a submodule, an essential submodule, or a direct summand of $M$, respectively. End$_{R}(M)$ is the ring of
$R$-endomorphisms of $M$ and E$(M)$ denotes the injective hull of $M$. The notations  $M^{(A)}$ and $M^{A}$ mean $\oplus_{i\in A}M_{i}$  and $\prod_{i\in A} M_{i}$, respectively,  where $A$ is an index  set and each $M_i\simeq M$.  The  annihilator of an element $m\in M$ will be denoted by ann$_{R}(m)$. ${\rm J}(R)$ stands for the Jacobson radical of a ring $R$.
For any $n\in\Bbb{N}$, ${\rm M}_{n}(R)$ stands for $n\times n$ matrix ring over a ring $R$.    The singular submodule of $M$ is denoted by ${\rm Z}(M)$ and the second singular submodule of $M$ is denoted  by ${\rm Z}_{2}(M)$.  The cardinal number of a set $X$ is denoted by $|X|$ and  the cardinal number of the set $\mathbb{N}$ of natural numbers is customarily denoted by $\aleph_{\circ}$.   For other  terminology and  results, we refer   the reader to   \cite{Extension},  \cite{Kaplansky}  and \cite{Muller}.  


\section{\bf  On Schr\"{o}der-Bernstein property for Baer modules}
A ring $R$ with an involution $\star$ is called {\it a Baer $\star$-ring} if the right annihilator of every nonempty subset of $R$ is generated by an idempotent. We remind that an $R$-module $M$ is called {\it a Baer module} if ann$_{S}(N)$ is a direct summand of $S$ where $S={\rm End}_{R}(M)$ and $N\leq M$ (or equivalently, for every left ideal $I$ of $S$, ann$_{M}(I)$ is a direct summand of $M$).  A ring $R$ is called {\it Baer} if $R_R$ is Baer. In a  commutative domain $R$,  every right ideal is  Baer, as a module over $R$. \\
By Kaplansky's Theorem, it is  known that every Baer $\star$-ring satisfies the SB property. Since Baer $\star$-rings are Baer rings, it is natural to ask whether Baer rings do satisfy the SB property. So  we will be concerned with the question of when any two Baer modules which are subisomorphic or direct summand subisomorphic to  each other are necessarily isomorphic.\\

Consider the following conditions on an $R$-module $M$:\\
\noindent ($C_1$)  Every submodule of $M$ is essential in a direct summand of $M$.\\
\noindent ($C_2$) Every submodule isomorphic to a direct summand of $M$ is a direct
summand.\\
 \noindent ($C_3$) The sum of two independent direct summands of $M$ is again a
direct summand.

The module  $M$ is called  {\it continuous} if it has C$_1$ and C$_2$, {\it  quasi-continuous} if it has  C$_1$ and C$_3$ and  {\it extending} if it has C$_1$,  respectively \cite[Chapter 2]{Muller}.
It is well known that $C_2\Rightarrow C_3$. The following implications hold:
 \begin{center}
 Injective $\Rightarrow$ Quasi-injective $\Rightarrow$ Continuous $\Rightarrow$ Quasi-continuous $\Rightarrow$ Extending.
 \end{center}

Recall from \cite[1.1.8]{Extension}, that
the singular submodule Z$(M)$ of an $R$-module $M$ is defined by
Z$(M) = \{m \in  M \ |\ mI = 0$ for some essential right ideal
$I$ of $R \}$ and the submodule Z$_2(M)$ satisfying   Z$(M/{\rm
Z}(M))$ = Z$_2(M)/{\rm Z}(M)$ is called {\it the second singular
submodule of $M$}. The module $M$ is called {\it singular}
({\it nonsingular}) if Z$(M)=M$ (Z$(M)=0$). A ring $R$ is called {\it right nonsingular} if it is nonsingular as a right $R$-module.
An $R$-module $M$ is called {\it $\mathcal{K}$-nonsingular} if for any $\varphi\in S={\rm End}_{R}(M)$, ${\rm Ker}\varphi\leq_{ess}M$ implies $\varphi=0$. Every nonsingular module is $\mathcal{K}$-nonsingular.
We begin with a result from \cite{Extension} for latter uses.


\begin{thm}\cite[Lemma 4.1.16]{Extension}\label{extending nonsingular}
Every $\mathcal{K}$-nonsingular extending  module is Baer. In particular, every nonsingular extending module is Baer. 
\end{thm}


In the following, we give some   examples  to show that any two subisomorphic Baer modules are not isomorphic in general.


\begin{exam}\label{1.1}
Let $R$ be a commutative domain and $I$ be any non principal ideal of $R$. Clearly $R$,  $I$ are Baer $R$-modules and subisomorphic to each other however  $R\not\simeq I$.
\end{exam}


In the following  we show that even if $N$ and $K$  are  Baer $R$-modules with the stronger condition: ``$N$ is isomorphic to a submodule of $K$ and $K$ is isomorphic to a direct summand of $N$" then   $N$ is not isomorphic to $K$ in general.


\begin{exam}\label{Baer.subisomorphic}
Let $N=\Bbb{Q}^{(\Bbb{N})}\oplus \Bbb{Z}$ and $K=\Bbb{Q}^{(\Bbb{N})}$.
 Thus the $\Bbb{Z}$-modules  $N$ and $K$ are nonsingular extending
 \cite[p. 19]{Muller}. Therefore by Theorem \ref{extending nonsingular}, $N$ and $K$ are Baer. Moreover, it is clear that $K\leq_{\oplus} N$ and $N$ is isomorphic to a submodule of $K$, however, $N$ is not isomorphic to $K$.
\end{exam}


We recall that an $R$-module $M$ is called {\it Rickart} if ${\rm Ker}f={\rm ann}_{M}(f)\leq_{\oplus}M$ where $f\in {\rm End}_{R}(M)$. A ring $R$ is called {\it right (left) Rickart} if $R_R$ (${}_{R}R$) is Rickart. Clearly every Baer module is Rickart. For more details see \cite[Chapter 3]{Extension}. \\
We recall that a ring $R$ is called {(\it von-Neumann) regular} provided that for each $r\in R$, $r\in rRr$. It is well known that   regular rings $R$ are precisely the ones whose  every principal (finitely generated) right ideals are direct summands. The following result was shown in \cite{Rangaswamy}:

\begin{thm}\cite[Theorem 4]{Rangaswamy}\label{Rangaswamy}
Let $M$ be an $R$-module and $S={\rm End}_{R}(M)$. Then $S$ is a  regular ring if and only if for each $\varphi\in S$, ${\rm Ker}\varphi$ and ${\rm Im}\varphi$ are direct summands of $M$.
\end{thm}


By the above theorem, any   $R$-module $M$,  with the regular endomorphism ring is a Rickart module.\\
 Following  \cite[Chapter 4]{Muller}, an  $R$-module $M$ is said to have {\it $D_{2}$ property} whenever  $M/N$ is isomorphic to a direct summand of $M$ implies that  $N$ is a direct summand of $M$ where $N\leq M$. It is well known that every quasi-projective module has ${\rm D}_{2}$ property. Moreover, every  Rickart module has  ${\rm D}_2$ property \cite[Theorem 1.5]{dual rickart}.\\
Following \cite{S.B}, an $R$-module $M$ is said to  {\it satisfy  the Schr\"{o}der-Bernstein property (or SB property)} if any two d-subisomorphic direct summands of $M$ are isomorphic (the $R$-modules $N$ and $K$ are called {\it d-subisomorphic to each other}  whenever $N$ is isomorphic to a direct summand of $K$ and $K$ is isomorphic to a direct summand of $N$). A ring $R$ is said to satisfy the SB property if $R_R$ has the SB property. We note that the notion of the SB property for rings is right and left symmetric \cite[Theorem 2.4 (c)]{S.B}.\\
 In the following lemma, we show that  Rickart modules $N$ and $K$ are  d-subisomorphic to each other if and only if they are epimorphic images of each other.


 \begin{lem}\label{dual S.B and Baer}
 Let $N$ and $K$ be Rickart modules. Then $N$ and $K$ are d-subisomorphic to each other if and only if there are $R$-epimorphisms $N\rightarrow K$ and $K\rightarrow N$.
 \end{lem}
 \begin{proof}
 We note that every Rickart module has D$_{2}$ property \cite[Theorem 1.5]{dual rickart} and for modules $M_1$ and $M_2$ with D$_2$ property, $M_1$ and $M_2$ are d-subisomorphic to each other if and only if they are epimorphic images of each other  \cite[Lemma 2.4]{DSB}.
 \end{proof}

In the next  example, we show that two Rickart modules which are
d-subisomorphic
to each other are not isomorphic in general.

\begin{exam}\label{2}
Suppose that  $V$ is  an infinite dimensional vector space over a field $F$ with $S= {\rm End}_{F}(V)$. Let  $\beta=\{v_{i}\}_{i\in I}$ be a basis for $V_{F}$ and $R:=\{(f,g)\in S\times S \  | $ rank$(f-g)< \infty \}$. Clearly $R$ is a subring of $S\times S$. We note that  $R$ is a regular ring and so by Theorem \ref{Rangaswamy}, $R_R$ is Rickart.
There exist idempotents $e$ and $ g$  in $R$ such that $eR$ and $gR$ are d-subisomorphic to each other however $eR$ is not isomorphic to $gR$ (see \cite[Example 2.2]{S.B} for more details).
Since every direct summand of  a Rickart  module has the property \cite[Proposition 4.5.4]{Extension}, $eR$ and $gR$ are Rickart $R$-module.  Therefore the Rickart module $R_R$ does not satisfy the SB property.
\end{exam}


Regarding examples \ref{1.1}, \ref{Baer.subisomorphic},  \ref{2}, and Theorem \ref{Kap.Baer} about Baer $\star$-rings, it is natural to ask the question: {\it ``does any Baer module satisfy the SB property?"}\\
It is clear that  any Baer module satisfies the SB property if and only if any pair of Baer modules which are d-subisomorphic to each other are isomorphic.
In order to answer this question, we note that  the main point in the proof of Theorem \ref{Kap.Baer},  is that the set of all projections in a Baer $\star$-ring forms  a complete lattice under  $``\leq"$ (if $e, f$ are idempotents in a ring $R$, we write $e\leq f$ in case $ef=fe=e$, i.e., $e\in fRf$). While in a Baer ring, the set of all  right ideals generated by idempotents forms a complete lattice \cite[Theorem 3.1.23]{Extension}.  In the next example we show that  the set of all idempotents in a Baer ring  is not   a complete lattice in general.


\begin{exam}\label{triangular metrix ring}
Suppose that  $R={\rm M}_{3}(\Bbb{Z})$ is  the $3\times 3$ matrix ring over $\Bbb{Z}$.  Then  $R$ is a Baer ring  \cite[Theorem 6.1.4]{Extension}. Let   $e_{x}=\left[\begin{array}{ccc} 1 &0&x \\0&0&0\\0&0&0
\end{array}\right]$ where $x\in\Bbb{Z}$. Clearly each  $e_{x}$ is an idempotent in $R$. We show that the set E $=\{e_{x} \ | \ x\in \Bbb{Z}\}$ has no  supremum in the lattice of all idempotents in $R$. To see this, let Sup$(E)=f$ where $f=\left[\begin{array}{ccc} a_{1} &b_{1}&c_{1} \\ a_{2}& b_{2}& c_{2}\\a_{3}& b_{3}&c_{3}
\end{array}\right]$ be an idempotent in $R$.
 Thus $e_{x}\leq f$ for every  $x\in\Bbb{Z}$ and so $fe_{x}=e_{x}f=e_{x}$. This shows that $a_{1}=c_{3}=1$, $b_{1}=c_{1}=a_{2}=a_{3}=b_{3}=0$ and  $f=\left[\begin{array}{ccc} 1 &0&0 \\0&b_{2}&c_{2}\\0&0&1
\end{array}\right]$. Since $f^{2}=f$,  $b_{2}^{2}=b_{2}$ and $b_{2}c_{2}+c_{2}=c_{2}$. Therefore either  $b_{2}=0$ or $b_{2}=1$.  If $b_{2}=1$  then $c_{2}=0$. Thus either $f=\left[\begin{array}{ccc} 1 &0&0 \\0&0&c_{2}\\0&0&1
\end{array}\right]$ or $f=\left[\begin{array}{ccc} 1 &0&0 \\0&1&0\\0&0&1
\end{array}\right]$. Let $f_{1}:=\left[\begin{array}{ccc} 1 &0&0 \\0&1&0\\0&0&1
\end{array}\right]$ and $f_{2}:=\left[\begin{array}{ccc} 1 &0&0 \\0&0&c_{2}\\0&0&1
\end{array}\right]$. Suppose that   $f=f_{2}$.  In this case, let  $g=\left[\begin{array}{ccc} 1 &0&0 \\0&0&\alpha \\0&0&1
\end{array}\right]$ where $c_{2}\neq\alpha\in\Bbb{Z}$.
Clearly  for every $x\in\Bbb{Z}$,   $e_{x}\leq g$.   Thus  $f_{2}\leq g$ and then  $c_{2}=\alpha$, a contradiction. Otherwise, $f=f_{1}$. Since for every $x\in\Bbb{Z}$, $e_{x}\leq f_{2}$, we have  $f_{1}\leq f_{2}$, a contradiction. Thus $E$ has no  supremum between the set of all idempotents in $R$. Hence the set of all idempotents of $R$ is not a complete lattice.
\end{exam}


By the following result from \cite{S.B}, any   module with  idempotents in its endomorphism ring forming a complete lattice  has the SB property.

\begin{thm}\label{complete lattice}\cite[Theorem 2.23]{S.B}
Let $M$ be  an  $R$-module and  $S={\rm End}_{R}(M)$. If the set of all idempotents in $S$ is a complete lattice with respect to the ordering $e\leq f$ then $M$ satisfies the  SB property.
\end{thm}


 In what follows, we show that under some certain conditions, any Baer module satisfies  the SB property.

\begin{thm}\label{reduced}
Let $N$ and $K$ be two  Baer $R$-modules which are d-subisomorphic to  each other. If for all idempotents $g, h\in {\rm  End}_{R}(N)$, $gh=0$ implies $hg=0$, then $N\simeq K$.  In particular every reduced Baer ring satisfies the SB property.
\end{thm}
\begin{proof} Without loss of generality, we may assume  that $K\leq_{\oplus} N$ and $N$ is isomorphic to a direct summand of $K$. By Theorem \ref{complete lattice}, it is enough to show that the set of all idempotents in $S=$ End$_{R}(N)$ forms  a complete lattice. Since $N$ is Baer, $S$ is a Baer ring \cite[Theorem 4.2.8]{Extension}.  Therefore  the set of all direct summands in $S_S$ forms  a complete lattice  \cite[Theorem 3.1.23]{Extension}.  Now let $\{e_{i}\}_{i\in I}$ be an arbitrary family of idempotents in $S$. Thus  there exists an idempotent $e\in S$ such that Sup$\{e_{i}S\}_{i\in I}=eS$. Thus  each  $e_{i}S\leq eS$. Hence every  $e_{i}=ee_{i}$ and so $(1-e)e_{i}=0$. Now by our assumption on $S$,  for all $i\in I$, $e_{i}(1-e)=0$ and so  $e_{i}\leq e$. Now let $g$ be an arbitrary idempotent of $S$ such that every $e_{i}\leq g$.  Hence $e_{i}S\leq gS$ and so $eS\leq gS$. Thus $(1-g)e=0$ and  so $e(1-g)=0$. Thus  $e\leq g$. Therefore Sup$\{e_{i}\}_{i\in I}=e$. Similarly every family of idempotents in $S$ has an infimum.   The last statement is now clear.
 \end{proof}


An $R$-module $M$ is called {\it a duo module } if  every submodule of $M$ is fully invariant.  A ring $R$ is called {\it a  right (or left) duo ring} if $R_R$ (or ${}_{R}R$) is duo. For more details, see \cite{duomodule}.


\begin{cor}\label{duo.Baer.S.B}
 Every  duo Baer module  satisfies the SB property. In particular every right (or left) duo  Baer ring satisfies the SB property.
\end{cor}
\begin{proof} Let $R$ be a ring and $M$ be a duo Baer $R$-module. Suppose that $e, f\in S=$ End$_{R}(M$) are idempotents such that $ef=0$. By our assumption,  $e(M)$ is a fully invariant submodule of $M$. Thus $f(e(M))\leq e(M)$ and so $0=efe=fe$. An application of
 Theorem \ref{reduced} yields  the result.
 The last statement is now clear.
\end{proof}


\begin{cor}\label{omm.Baer}
Every commutative Baer ring satisfies the SB property.
\end{cor}
\begin{proof}
It follows from Corollary \ref{duo.Baer.S.B}.
\end{proof}


By the next result from \cite{S.B}, every module with ascending chain condition on its direct summands satisfies the SB property.

\begin{thm}\cite[Theorem 2.16]{S.B}\label{sim.direct summand}
Let $M$ be an $R$-module with the following condition:\\
for every descending chain $M_{1}\geq M_{2}\geq...$ of direct summands of $M$, there exists $n\geq 1$ such that  $M_{n}\simeq M_{n+1}$. \\
Then $M$ satisfies the SB property. In particular, every module with descending (ascending) chain condition on direct summands  satisfies the SB property.
\end{thm}


\begin{thm}\label{Baer.Countable}
If $M$ is a  Baer $R$-module with only countably many direct summands, then $M$ satisfies the SB property.
\end{thm}
\begin{proof}
 Let $M$ be a Baer $R$-module with only countably many direct summands.  Thus $S=$ End$_R(M)$ is a Baer ring \cite[Theorem 4.2.8]{Extension} and it has  only countably many idempotents. This  implies that  $S$ has no infinite set of orthogonal idempotents  \cite[Theorem 3.1.11]{Extension}. Therefore  $M$ has descending  chain condition on its direct summands and so by Theorem \ref{sim.direct summand}, $M$ satisfies the SB property.
 \end{proof}


In the following some applications of our results are indicated by characterizing rings over which any pair of  subisomorphic Baer  modules are isomorphic.


\begin{prop}\label{Baer. subisomorphic}
Let $R$ be a ring and any pair of  subisomorphic Baer $R$-modules  are isomorphic.
Then  every  nonsingular extending  $R$-module is injective.
\end{prop}
\begin{proof} Suppose that any two subisomorphic Baer $R$-modules  are isomorphic and $M$ is   any  nonsingular extending  $R$-module. Let $K=$E$(M)^{\Bbb{N}}\oplus M$ and $N=$ E$(M)^{\Bbb{N}}$. The $R$-module $N$ is extending and nonsingular module.  By Theorem \ref{extending nonsingular},   $N$ is a Baer module. Moreover  by   \cite[Theorem 4.2.18] {Extension}, $K$  is also  a Baer module. Clearly $N$ is a direct summand of $K$  and $K$ is isomorphic to a submodule of $N$. Thus by our assumption, $N\simeq K$ and so $M$ is an injective $R$-module.
\end{proof}


We recall that an $R$-module $M$ {\it has finite unifrom dimension} if there exist  uniform submodules $U_{1}, U_{2},..., U_{n}$ of $M$ such that $\oplus_{i=1}^{n}U_{i}\leq_{ess}M$. It is well known that $M$ has finite uniform dimension if and only if $M$ does not contain any infinite direct sums. Clearly every uniform module  has finite uniform dimension.


\begin{thm}\label{Z2}
Let $R$ be a ring such that $R_R$ has  finite uniform dimension. Consider the following conditions:\\
{\rm (a)} Any pair of  subisomorphic Baer $R$-modules are isomorphic;\\
{\rm (b)} Any pair of Baer $R$-modules $N$ and $K$ are isomorphic provided that $N$ is isomorphic to a submodule of $K$ and $K$ is isomorphic to a direct summand of $N$;\\
{\rm (c)} $R/{\rm Z}_{2}(R)$ is a semisimple Artinian ring.\\
Then (a) $\Rightarrow$ (b) $\Rightarrow$ (c).
\end{thm}
\begin{proof}
(a) $\Rightarrow$ (b). This  is clear.\\
(b) $\Rightarrow$ (c). Let $I=$ ${\rm Z}_{2}(R_R)$.  Since $R$ has finite uniform dimension and  $I$ is a closed submodule of $R_{R}$,  it is well known that  $R/I$ has finite uniform dimension as an  $R/I$ module.  In addition, Z$((R/I)_{R})=0$, then $R/I$ is a right  nonsingular ring. We also note  that if $M$ is an $R/I$-module, then $M$ is Baer as $R/I$-module  if and only if it is  Baer as $R$-module. It is routine to see that  the condition (b) holds  for $R/I$-modules. Thus without loss of generality we may assume that  $R$ is a right nonsingular ring.  By our assumption on $R$, there exist uniform right ideals $U_{1}, U_2,...,U_{n}$ ($n\geq 1$) of $R$, such that $U_{1}\oplus U_{2}\oplus...\oplus U_{n}\leq_{ess} R_R$. Since each $U_i$ is a nonsingular extending $R$-module,  by  the proof of Proposition \ref{Baer. subisomorphic} and the assumption (b),  each $U_{i}$ must be injective. Hence  $U_{1}\oplus U_{2}\oplus...\oplus U_{n}$ is an injective $R$-module and  so  $R=U_{1}\oplus U_{2}\oplus...\oplus U_{n}$ is a right self-injective ring. Therefore   Z$(R_R)=$ J$(R)$ and $R/$J$(R)$ is a   regular ring \cite[Proposition 3.15]{Muller}. Since Z$(R_R)=0$,  $R$ is a  regular ring. Therefore each $U_{i}$ is  a simple $R$-module and so  $R$ is a semisimple Artinian ring, as desired.
\end{proof}


\begin{cor}\label{nonsingular}
Let $R$ be a right nonsingular ring and  $R_R$  has    finite uniform dimension.  The following conditions are equivalent: \\
{\rm (a)} Any pair of  subisomorphic Baer $R$-modules are isomorphic;\\
{\rm (b)} Any pair of Baer $R$-modules $N$ and $K$ are isomorphic provided that $N$ is isomorphic to a submodule of $K$ and $K$ is isomorphic to a direct summand of $N$;\\
{\rm (c)} $R$ is a semisimple Artinian ring.
\end{cor}
\begin{proof}
It follows   from  Theorem  \ref{Z2} and the fact that any pair of semisimple subisomorphic modules are isomorphic.
\end{proof}

We remind that by Bumby's Theorem, any two subisomorphic  injective modules  are  isomorphic to each other \cite[Theorem]{bumby}.  Therefore if every Baer module is injective then any two subisomorphic Baer modules are isomorphic to each other.
So  in the following,  we investigate the stronger case: ``when    every Baer $R$-module  is injective".\\
We recall that a ring $R$ is called {\it a right  V-ring} if every simple $R$-module is injective.


\begin{thm}\label{all Baer is inj}
Let $R$ be a ring. The following are equivalent: \\
{\rm (a)} Every Baer $R$-module is injective; \\
{\rm (b)} $R$ is a semisimple Artinian ring.
\end{thm}
\begin{proof}
 (a) $\Rightarrow$ (b). First we claim  that $R$ is a right Noetherian right V-ring provided that
 every semisimple $R$-module is injective. To see this, assume that every semisimple $R$-module is injective. Clearly $R$ is a right V-ring.  It is well known that a ring $R$  is  right Noetherian  if and only if  any arbitrary direct sum of cocyclic injective $R$-modules is injective. 
 We note that cocyclic injective $R$-modules  are  precisely injective envelope of  simple $R$-modules.
 Let $\{T_{i}\}_{i\in I}$ be a family of simple $R$-modules. Since $R$ is a  right V-ring,  $\oplus_{i\in I}$E$(T_{i})=\oplus_{i\in I}T_{i}$  is  semisimple and so by our assumption is an injective $R$-module. Thus $R$ is right Notherian. \\
  Now suppose that every Baer $R$-module is injective. Thus   every semisimple $R$-module is  injective and then  by the above note,  $R$ is a right Noetherian right V-ring. Therefore $R$ is  a semiprime right Goldie ring and so  $R$ is right nonsingular.  Since every Baer $R$-module is injective, then any pair of  subisomorphic  Baer $R$-modules are  isomorphic to each other. Now by Corollary \ref{nonsingular}, $R$ is a semisimple Artinian ring.\\
  (b) $\Rightarrow$ (a). It is obvious.
\end{proof}


By Theorem \ref{extending nonsingular},  every extending  nonsingular module is  Baer. Therefore  the question  ``{\it when any pair of  subisomorphic extending modules are isomorphic to each other}" is natural.  In  Example \ref{Baer.subisomorphic},  we found two subisomorphic nonsingular extending $\Bbb{Z}$-modules $N$ and $K$ which  are not isomorphic to each other. While over a right Noetherian ring $R$, any mutually d-subisomorphic extending $R$-modules are isomorphic to each other \cite[Theorem 2.12]{S.B}.\\
In the next,  we show  that the study  of subisomorphic extending modules leads to the study of subisomorphic singular extending modules and subisomorphic nonsingular extending modules.


\begin{thm}\label{separate extendings}
Any two (d-)subisomorphic extending $R$-modules are isomorphic to each other if and only if any pair of (d-)subisomorphic extending singular $R$-modules are isomorphic to each other and any pair of (d-)subisomorphic extending nonsingular $R$-modules are isomorphic.
\end{thm}
\begin{proof}
First we note that  $M$ is an extending $R$-module if and only if  $M={\rm Z}_{2}(M)\oplus M_{1}$ where  ${\rm {Z}}_{2}(M)$ and $M_{1}$ are extending  submodules  and ${\rm Z}_{2}(M)$ is $M_1$-injective.
Suppose that any pair of subisomorphic singular extending $R$-modules are isomorphic to each other and  it does hold  true  for the class of  nonsingular extending $R$-modules.
Let $X$ and $Y$ be subisomorphic extending $R$-modules. Without loss of generality, we may assume that $Y\leq X$ and $\theta:X\rightarrow Y$ is an $R$-monomorphism. Moreover, $X={\rm Z}_{2}(X)\oplus X_{1}$ and $Y={\rm Z}_{2}(Y)\oplus Y_1$ where $X_{1}\leq X$ and $Y_1 \leq Y$. It is clear that ${\rm Z}_{2}(Y)\leq {Z_{2}}_(X)$ and $\theta({\rm Z}_{2}(X))\leq {\rm Z}_{2}(Y)$. Therefore by our assumption, ${\rm Z}_{2}(X)\simeq {\rm Z}_{2}(Y)$. On the other hand, since $Y_{1}  \cap  {\rm Z}_{2}(X)=0$ and $\theta(X_{1}) \cap  {\rm Z}_{2}(Y)=0$, then $Y_{1}\leq X_1$ and $X_{1}\simeq \theta(X_{1})\leq Y_{1}$. Thus $X_{1}$and $Y_{1}$ are nonsingular extending $R$-modules which are subisomorphic to each other. Therefore by our assumption $X_{1}\simeq Y_{1}$. Hence $X\simeq Y$.  The converse  is clear. \\
The proof of the case of d-subisomorphism is similar.
\end{proof}

\begin{thm}\label{extenging.nonsingular.charac}
Let $R$ be a ring. The following statements are equivalent:\\
{\rm (a)} For any $R$-module $Y$ and any  nonsingular extending (injective) $R$-module $X$, if $X$ and $Y$ are subiomorphic to each other then $X\simeq Y$; \\
{\rm (b)} Any pair of nonsingular subisomorphic $R$-modules are isomorphic to each other;\\
{\rm (c)} $R/{\rm Z}_{2}(R)$ is a semisimple Artinian ring.
\end{thm}
\begin{proof}
First we note that  $R/{\rm Z}_{2}(R)$ is a semisimple Artinian ring if and only if every nonsingular $R$-module is injective \cite[Theorem 3.5]{S.B}.\\
(a) $\Rightarrow$ (c). Let  (a) holds true.   Suppose that  $M$ is any  nonsingular $R$-module and   $Y:={{\rm E}(M)}^{\Bbb{N}}\oplus M$ and $X:={{\rm E}(M)}^{\Bbb{N}}$.  It is clear that $X$,   $Y$ are subisomorphic nonsingular $R$-modules and  $X$ is  injective (and so extending).  Thus  by  our assumption,  $X\simeq Y$  and then     $M$ is an injective $R$-module.  Therefore   every nonsingular $R$-module is injective and so   $R/{\rm Z}_{2}(R)$ is  a semisimple Artinian ring.\\
(c) $\Rightarrow$ (b).  It follows by the above note and Bumby's Theorem. \\
(b) $\Rightarrow$ (a). We  note that the class of nonsingular modules is closed  under  taking  submodules. 
\end{proof}


In \cite[Theorem 4.7]{S.B}, it has been shown that {``\it over a commutative domain $R$,  any two subisomorphic uniform $R$-modules  are isomorphic to each other if and only if $R$ is  a PID"}. We note that every uniform module is extending. In the following, we characterize a commutative domain  $R$ over which any pair of  subisomorphic extending $R$-modules are isomorphic to each other.
First we include  some results from  \cite{extending over comm. domain} and \cite{structure}  that characterize the structure of extending modules over Dedekind domains.\\

Let $M$ be  an $R$-module. We recall that  $M$ is said to be  {\it  prime} if $M$ is fully faithful as module over $R/{\rm ann}_{R}(M)$. It is routine to check that the annihilator of every prime $R$-module $M$ is a  prime ideal. {\it An associated prime of $M$} is any  prime ideal  $P$ of $R$ which equals the annihilator
of some prime submodule $N$ of $M$.  The set of all associated primes of $M$ is denoted by {\it  Ass$(M)$}. For any $R$-module $M$, ${\rm Ass}(M)\neq \emptyset$ provided that $R$ is right Noetherian. It is well known that over a right Noetherian ring $R$,  any uniform module has unique associated prime. For more details see \cite[Chapters 3 and 5]{Goodearl}.


\begin{thm}\cite[Theorem 18]{structure}\label{structure}
Let $R$ be a commutative Noetherian ring. An $R$-module $M$ is extending if and only if $M=\oplus_{P}M(P)$ (unique up to isomorphism) where $M(P)$ has
associated prime $P$, is extending and is $M(Q)$-injective for all $Q\neq P$.
\end{thm}


\begin{thm}\cite[Corollary 23]{structure}\label{21}
Let $M$ be a torsion module over  a Dedekind-domain $R$. Then $M$ is extending if and only if for each nonzero prime ideal $P$, either $M(P)$ is injective or $M(P)$ is a direct sum of copies $R/P^{n}$ or $R/P^{n+1}$ for some $n\in \Bbb{N}$.
\end{thm}


\begin{thm}\cite[Theorem 14]{extending over comm. domain}\label{30}
Let $M$ be a torsion free module over a Dedekind domain $R$. Then $M$ is extending if and only if $M=F\oplus E$ where $E$ is an injective $R$-module and $F\simeq \oplus_{i=1}^{n}NI_{i}$ where $N$ is a proper $R$-submodule of the quotient field $K$  and each $I_{i}$  is a  fractional ideal of $R$.
\end{thm}


In the next Theorem, we investigate when over a commutative Noetherian ring $R$, two subisomorphic extending $R$-modules are isomorphic.


\begin{prop}\label{10}
Let $R$ be a commutative Noetherian ring. The following statements  are equivalent:\\
{\rm (a)} Any pair of subisomorphic extending $R$-modules are isomorphic;\\
 {\rm (b)} For any prime ideal $P$ of $R$, any pair of subisomorphic extending $R$-modules with associated prime  $P$ are isomorphic to each other.
\end{prop}
\begin{proof}
(a) $\Rightarrow$ (b). This is clear. \\
(b) $\Rightarrow$ (a). Assume that $X$ and $Y$ are subisomorphic extending $R$-modules and $f:X\rightarrow Y$, $g:Y\rightarrow X$ are $R$-monomorphisms. By Theorem \ref{structure},  $X$ and $Y$ have the decompositions $X=\oplus_{P}X(P)$ and $Y=\oplus_{P}Y(P)$ where $P$ runs over all nonzero prime ideals,  $X(P)$ and $Y(P)$ are extending submodules of $X$ and $Y$ respectively, with associated primes $P$. We claim that for any prime ideal $P$,  $f(X(P))\cap \oplus_{P\neq Q}Y(Q)=0$. Let  $L:=f(X(P))\cap \oplus_{P\neq Q}Y(Q)\neq 0$.
Over a right Noetherian ring $R$, any  nonzero $R$-module has an associated prime \cite[Proposition 3.12]{Goodearl}. Therefore  $\emptyset\neq {\rm Ass}(L)\subseteq {\rm Ass}(X(P))=\{P\}$. Thus
${\rm Ass}(L)=\{P\}$. On the other hand $\{P\}={\rm Ass}(L)\subseteq {\rm Ass}(\oplus_{P\neq Q}Y(Q))=\{Q \ | \ Q$ is a prime ideal of $R$ and $Q\neq P\}$, a contradiction. Therefore $f(X(P))\cap \oplus_{P\neq Q}Y(Q)=0$ and then $f(X(P))\subseteq Y(P)$. Similarly, $g(Y(P))\subseteq X(P)$. Thus for any prime ideal $P$, $X(P)$ and $Y(P)$ are extending subisomorphic $R$-modules with associated primes $P$. Hence by our assumption, for any prime ideal $P$, $X(P)\simeq Y(P)$ and so $X$ is isomorphic to $Y$.
\end{proof}



\begin{prop}\label{12}
let $R$ be a commutative domain. The following are equivalent:\\
{\rm (a)} Any pair of subisomorphic extending $R$-modules are isomorphic.\\
{\rm (b)} Any pair of subisomorphic torsion free extending $R$-modules are isomorphic.\\
{\rm (c)} $R$ is a field.
\end{prop}
\begin{proof}
(a) $\Rightarrow$ (b) and (c) $\Rightarrow$ (a) are clear.\\
(b) $\Rightarrow$ (c). Since $R$ is a commutative domain, any nonzero ideal of $R$ is uniform and so extending.  Let $I$ be any nonzero ideal of $R$. Therefore $R$ and $I$ are two extending torsion free $R$-modules which are subisomorphic to each other. Hence by our assumption, $R\simeq I$. Thus $R$ is a PID.  Now let $X=Q^{\Bbb{N}}\oplus R$ and $Y=Q^{\Bbb{N}}$ where $Q=Q(R)$ is the  quotient field of $R$. By Theorem \ref{30}, $X$ and $Y$ are extending. It is clear that $X$ and $Y$ are also torsion free $R$-modules which are subisomorphic to each other. Therefore by our  assumption,  $X\simeq Y$. Thus  $R_R$ should be injective and then $R=Q$ is a field.
\end{proof}


In the following example we show that two mutually subisomorphic torsion extending $\Bbb{Z}$-modules are not isomorphic to each other in general.

\begin{exam}\label{13}
Suppose that $X=\Bbb{Z}_{p}\oplus \Bbb{Z}_{p^{2}}^{(\Bbb{N})}$ and $Y=\Bbb{Z}_{p}^{(2)}\oplus \Bbb{Z}_{p^{2}}^{(\Bbb{N})}$ where $p$ is a prime number.  We note that $X$ and $Y$ are  torsion and also by Theorem \ref{21}, are  extending   $\Bbb{Z}$-modules. It is obvious that $X$ and $Y$ are subisomorphic to each other. However, $X$ is not isomorphic to $Y$. To see this, it is routine to check that $X$ has only one simple submodule  $\Bbb{Z}_{p}\oplus 0$ which is not contained in any proper cyclic submodule, however $Y$ has two simple submodules $\Bbb{Z}_{p}\oplus 0$ and $0\oplus \Bbb{Z}_{p}$ which are  not contained in any proper cyclic submodule.
Therefore two subisomorphic extending torsion modules are not necessarily  isomorphic to each other even over a PID.
\end{exam}


Now, in the next result we investigate when over a commutative principal ideal ring $R$, mutually subisomorphic torsion extending $R$-modules are isomorphic to each other.
We recall that a commutative ring $R$ {\it has zero dimensional} whenever every prime ideal of $R$ is maximal and is denoted by dim$(R)=0$.


\begin{thm}\label{15}
Let $R$ be a commutative principal ideal ring. If any pair of subisomorphic extending  torsion $R$-modules are isomorphic then ${\rm dim}(R)=0$.
\end{thm}
\begin{proof}
First we prove the case that $R$ is a PID. Assume that $R$ is a PID such that any pair of subisomorphic extending torsion  $R$-modules are isomorphic to each other. We show that $R$ is  a field. Let $Q$ be any nonzero prime ideal of $R$, $X=R/Q\oplus ({R/Q^{2}})^{(\Bbb{N})}$ and $Y=(R/Q)^{(2)}\oplus ({R/Q^{2}})^{(\Bbb{N})}$. Clearly $X$ and $Y$ are torsion $R$-modules. Since $R$ is a PID, then  $R/Q$ is isomorphic to a submodule of $R/Q^{2}$ and so $X$ and $Y$ are subisomorphic to each other. Moreover,  by Theorem \ref{21},  $X$ and $Y$ are extending $R$-modules. Therefore by our assumption $X\simeq Y$, a contradiction. This proves that $R$ has no nonzero prime ideal. Hence  the zero ideal of $R$ is maximal and so $R$ is a field, as desired. \\
Now we prove the case that $R$ is a principal ideal ring.
Let $P$ be any prime ideal of $R$ and  $M$ be any $R/P$-module. It is routine to see that ${\rm Z}(M_{R/P})\subseteq {\rm Z}(M_{R})$. Therefore if $M_{R/P}$ is torsion then $M_R$ is so. Besides,  $M_R$ is extending if and only if $M_{R/P}$ is extending. Since  $R/P$ is a PID and any pair of subisomorphic extending torsion $R/P$-module are isomorphic to each other, by the above note, $R/P$ is a field. Hence  $P$ is a maximal ideal of $R$.
\end{proof}


\begin{cor}\label{16}
Let $R$ be a PID. Then any pair of subisomorphic extending torsion $R$-modules are isomorphic if and only if $R$ is a field.
\end{cor}
\begin{proof}
It follows from Theorem  \ref{15}.
\end{proof}





We present in the following, a result from \cite{S.B} which will be used in this paper.
Let $M$ be an $R$-module. A  family  $\{M_{i}\}_{i\in\Lambda}$ of submodules of $M$ is called {\it a local direct summand} of $M$,  if $\sum_{i\in\Lambda} M_{i}$ is direct and $\sum_{i\in F}M_{i}$ is a direct summand of $M$ for every finite subset $F\subseteq \Lambda$. Further,  if  $\sum_{i\in\Lambda} M_{i}$ is a direct summand of $M$, we say that the {\it local direct summand is a direct summand} \cite[ Definition 2.15]{Muller}.


\begin{thm}\label{sum of chain of direct summand}\cite[Theorem 2.10]{S.B}
Let $N$ and $K$ be   $R$-modules which are subisomorphic to direct summands of each other. If every local direct summand of $N$ is a direct summand  then $N\simeq K$.
\end{thm}


Let $M$ be an $R$-module.
  $M$ is called {\it a dual Baer module} if for any $R$-submodule $N$ of $M$, the right ideal D$(N)=\{f\in S \ | $ Im$f\subseteq N\}$ of $S$ is generated by an idempotent in $S$  where $S=$ End$_{R}(M)$ \cite{dual baer tutu}.   The module $M$ is called {\it dual Rickart} if for every $R$-homomorphism $f:M\rightarrow M$, Im$f$ is a direct summand of $M$. It has been shown that  $M$ is dual Baer  if and only if  it is dual Rickart and  the  sum of every family of direct summands of $M$ is a direct summand \cite[Theorem 1.7]{dual rickart}. In the following, we show that subisomorphisms between a dual Rickart and a dual Baer module leads to an isomorphism between them. First we prove the following basic Lemma which should be compared with Lemma \ref{dual S.B and Baer}.


\begin{lem}\label{40}
Let $N$ and $K$ be dual Rickart modules. Then $N$ and $K$ are subisomorphic to each other if and only if they are d-subisomorphic to each other.
\end{lem}
\begin{proof}
We note that every dual Rickart module has C$_{2}$ property \cite[Proposition 2.21]{dual rickart} and for modules $N$ and $K$ with C$_{2}$ property  the notions of subisomorphism and d-subisomorphism coincide.
\end{proof}

\begin{prop}\label{d.rickar and baer}
Let $N$ be a dual Baer $R$-module and $K$  be a  dual Rickart $R$-module.  If $N$ and $K$ are  subisomorphic to each other  then $N\simeq K$.
\end{prop}
\begin{proof} Let $N$ be dual Baer,  $K$ be dual Rickart and $N$, $K$  be  subisomorphic to each other. By Lemma \ref{40}, $N$ and $K$ are d-subisomorphic to each other.
 Since $N$ is dual Baer,  the  sum of every family of direct summands of $N$ is a direct summand \cite[Theorem 2.1]{dual baer tutu}. Therefore if  $\{N_i\}_{i\in I}$ is  a family of submodules of $N$ which  is a  local direct summand of $N$ then each $N_{i}$ is a direct summand of $N$. Thus $\sum_{i\in I} N_{i}$ is a direct summand of $N$. Hence every local direct summand of $N$ is a direct summand. Now the  result follows from Theorem
\ref{sum of chain of direct summand}.
\end{proof}


The following example shows that  the hypothesis of  Proposition \ref{d.rickar and baer},  cannot be weakened any further such as to assume that both $N$ and $K$ are
dual Rickart.


\begin{exam}\label{1}
First we note that a ring $R$ is  regular if and only if $R_R$ is dual Rickart. Moreover every direct summand of a dual Rickart module is dual Rickart \cite[Remark 2.2 and Proposition 2.8]{dual baer tutu}.
Now let $R$ be the regular ring of Example \ref{2}.
Therein we found two non-isomorphic direct summands $eR$ and $gR$ such that they are d-subisomorphic to each other,  where $e$ and $g$ are idempotents in $R$. By  the above note, $eR$ and $gR$ are also dual Rickart $R$-modules.
\end{exam}


Here the following natural question arises: ``are any pair of   dual Baer modules which are epimorphic images of each other,  isomorphic?" We end the paper with giving an example showing that the answer of this question is negative.
The following  example can be also  compared with Examples \ref{1.1} and \ref{Baer.subisomorphic}.


\begin{exam}\label{3}
We show that two dual Baer modules which are epimorphic images of each other are not isomorphic in general.
 To see it, first we note that a ring $R$ is right hereditary right Noetherian if and only if every injective  $R$-module is dual Baer \cite[Theorem 2.3]{dual baer tutu}. Now let $X=\Bbb{Q}^{\Bbb{N}}$ and $Y=\Bbb{Q}^{\Bbb{N}}\oplus \Bbb{Z}_{p^{\infty}}$ where $p$ is a prime integer number. By the above note $X$ and $Y$ are dual Baer $\Bbb{Z}$-modules. Clearly there are epimorphisms $X\rightarrow Y$ and $Y\rightarrow X$while $X$ is not isomorphic to $Y$.
\end{exam}


{\bf Acknowledgments.}
The authors express their gratitude to Mathematics Research Institute, the College of Arts and Sciences, the Office of International Affairs, OSU-Columbus and OSU-Lima, for their partial support of this research work.



\begin{thebibliography}{60}










\bibitem{Extension}
G. F. Birkenmeier, J. K. Park, S. T. Rizvi, {\it Extensions of Rings and Modules}, Springer New York Heidelberg Dordrecht London (2013).


\bibitem{bumby}
R. T. Bumby, ``Modules which are isomorphic to submodules of each other", {\it  Arch. Math.,} 16 (1965) 184-185.


\bibitem{cantor2}
I. G. Connell, ``Some ring theoretic Schr\"{o}der-Bernstein theorems", {\it Trans. Amer. Math. Soc.,} 132 (2) (1968)  335-351.


\bibitem{solution of kaplansky}
P. Crawley, ``Solution of Kaplansky's Test Problems for
primary abelian groups", {\it J. Algebra},  2 (1965) 413-431.


\bibitem{S.B}
N. Dehghani, F. E. Azmy,   S. T. Rizvi,  ``On the Schr\"{o}der-Bernstein property for Modules",
 {\it J. Pure Appl. Algebra,} 223 (1) (2019) 422-438.


\bibitem{DSB}
N. Dehghani, S. T. Rizvi,  ``When modules which are mutually epimorphic are isomorphic", Submitted.


\bibitem{the kaplansky test problem}
P. C. Eklof,   S. Shelah,  ``The Kaplansky Test Problems for $\Bbb{N}_{1}$-separable groups", {\it Proc. Amer. Math. Soc}.,   126 (7) (1998) 1901-1907.


\bibitem{Goodearl}
K. R. Goodearl, R. B. Warfield, {\it An introduction to non commutative Noetherian rings}, Second edition. London Mathematical Society Student Texts, 61. Cambridge University Press, Cambridge, (2004).


\bibitem{Ashish}
P. A. Guil-Asensio, B. Kalebo\~{g}az,  A. K. Srivastava, ``The Schr\"{o}der-Bernstein problem for modules", {\it J. Algebra}, 498 (2018) 153-164.


\bibitem{extending over comm. domain}

M. A. Kamal,  B. J. M\"{u}ller, ``Extending modules over commutative domains",  {\it Osaka J. Math}.,  25 (1988) 531-538.


\bibitem{structure}
 M. A. Kamal, B. J. M\"{u}ller, ``The structure of extending modules over Noetherian rings", {\it Osaka J. Math.},  25 (1988) 539-551.


\bibitem{infinite abelian groups}
I. Kaplansky, {\it Infinite Abelian Groups}, University of Michigan  Press,  Ann Arbor (1954).


\bibitem{Kaplansky}
I. Kaplansky, {\it Rings of Operators},  Benjamin, New York  (1968).




\bibitem{soo.lee}
D. S. Lee,  C. Park,  ``Notes on correct modules", {\it Comm. Korean Math. Soc.,}  11  (2) (1966) 205-301.


\bibitem{dual rickart}
G. Lee, S. T. Rizvi, C. S. Roman, ``Dual Rickart Modules", {\it Comm. Algebra}, 39 (11) (2011)  4036-4058.


\bibitem{Muller}
S. H,  Mohamed,   B. J. M\"{u}ller, {\it Continuous and Discrete Modules},  London Math.
Soc.,  LN 147, Cambridge Univ. Press, (1990).


\bibitem{Rizvi.Muller}
B. J. M\"{u}ller,   S. T. Rizvi,  ``On injective and quasi-continuous modules", {\it J. Pure  Appl. Algebra}, 28 (1983)  197-210. 


\bibitem{duomodule}
A. C. \"{O}zcan, A. Harmanci,  P. F.  Smith, ``Duo modules", {\it Glasgow Math. J.,} 48 (2006) 533-545.




\bibitem{rososhek2}
S. K. Rososhek,  ``Correctness of modules", {\it Iz. Vuz. Math.,} 22   (10) (1978) 77-82.

\bibitem{rososhek}
S. K. Rososhek,  ``Purely correct modules", {\it Uspekhi Mat. Nauk,}  33   (3) (1978) 176.


\bibitem{dual baer tutu}
D. K. T\"{u}t\"{u}nc\"{u}, R. Tribak, ``On dual Baer modules", {\it  Glasg. Math. J.,} 52 (2) (2010) 261-269.


\bibitem{I.kaplansky}
E. Sasiada, ``Negative solution of I. Kaplansky's first test problem for abelian groups and a problem of K. Borsuk concerning cohomology groups", {\it Bull. Acad. Polon. Sci. S\'{e}r. Sci. Math. Astronom. Phys., }   9 (1961) 331-334.



\bibitem{cantor1}
A. d. S. Napoli, D. Mundici,   M. N. Praha,   ``A Cantor-Bernstein Theorem for
$\sigma$-complete MV-algebras", {\it Czech. Math. J}.,  53  (2) (2003) 437-447.


\bibitem{Rangaswamy}
 K. M. Rangaswamy, ``Abelian groups with endomorphic images of special types", {\it J. Algebra}, 6 (1967)  271-280.


\bibitem{Functor}
V. Trnkova,  V. Koubek,  ``The cantor-Bernstein Theorem for functors", {\it Comment. Math. Univ. Carol}.,  14 (1) (1973) 197-204.


\bibitem{vasconcelos}
W. V. Vasconcelos, ``Injective endomorphisms of finitely generated modules", {\it Proc. Amer. Math. Soc.,}   25  (4) (1970) 900-901.


\bibitem{correct}
R. Wisbauer, ``Correct classes of modules", {\it Algebra Discrete Math}., 4 (2004) 106-118.





\end{thebibliography}
\end{document}